\documentclass[final,3p]{elsarticle}
\usepackage{geometry}
\usepackage{mathrsfs,enumerate,amssymb,amsmath,amsthm,color,lineno}
\usepackage[all]{xy}
\usepackage{latexsym}
\usepackage{amscd,verbatim}
\usepackage{graphicx}
\usepackage[colorlinks=true]{hyperref}
\usepackage{amssymb}

\theoremstyle{plain}
\newtheorem{theorem}{Theorem}
\newtheorem{lemma}{Lemma}
\newtheorem{proposition}{Proposition}
\newtheorem{corollary}{Corollary}
\theoremstyle{definition}
\newtheorem{definition}{Definition}

\numberwithin{equation}{section}
\newtheorem{remark}{Remark}
\journal{}

\begin{document}

\begin{frontmatter}
\title{\huge Characterizing idempotent nullnorms on bounded lattices\tnoteref{mytitlenote}}
\tnotetext[mytitlenote]{This work was supported by the National Natural Science Foundation of China
(No. 11601449), the Science and Technology Innovation Team of Education Department of
Sichuan for Dynamical System and its Applications (No. 18TD0013),
the Key Natural Science Foundation of Universities in Guangdong Province (No. 2019KZDXM027),
the Youth Science and Technology Innovation Team of Southwest Petroleum University for Nonlinear
Systems (No. 2017CXTD02), and the National Nature Science Foundation of China (Key Program)
(No. 51534006).}


\author[a1,a2,a3]{Xinxing Wu}
\address[a1]{School of Sciences, Southwest Petroleum University, Chengdu, Sichuan 610500, China}
\address[a2]{Institute for Artificial Intelligence, Southwest Petroleum University, Chengdu, Sichuan 610500, China}
\address[a3]{Zhuhai College of Jilin University, Zhuhai, Guangdong 519041, China}
\ead{wuxinxing5201314@163.com}

\author[a1]{Shudi Liang}
\ead{shudi$\_$liang@163.com}

\author[a4]{G\"{u}l Deniz \c{C}ayl{\i}\corref{mycorrespondingauthor}}
\cortext[mycorrespondingauthor]{Corresponding author}
\address[a4]{Department of Mathematics, Faculty of Sciences, Karadeniz Technical University, 61080 Trabzon, Turkey}
\ead{guldeniz.cayli@ktu.edu.tr}

%
%

\begin{abstract}
Nullnorms with a zero element being at any point of a bounded lattice are an important generalization
of triangular norms and triangular conorms. This paper obtains an equivalent characterization for the
existence of idempotent nullnorms with the zero element $a$ on any bounded lattice containing two distinct
elements incomparable with $a$. Furthermore, some basic properties for the bounded lattice containing
two distinct element incomparable with a are presented.
\end{abstract}
\end{frontmatter}


\section{Introduction}

Being the generalizations of triangular norms and triangular conorms, which
were introduced by Schweizer and Sklar \cite{SS1983}, t-operators and
nullnorms were introduced by Mas et al. \cite{MMT1999} and Calvo et al. \cite%
{CBF2001}, respectively. Then, Mas et al.~\cite{MMT2002} showed that both of
them coincide with each other on $[0,1]$. Simply speaking, an nullnorm is a
binary operation obtained from triangular norms and triangular conorms using
an ordinal sum structure with the zero element lying anywhere in $[0,1]$.
Because of its special structure, nullnorms on $[0,1]$ are useful in both
theory and applications (\cite%
{AlsFraSch06,DDR2008,D2004,D2015,KMP2000,MMT2002,MMT2002-2,QZ2005}).
Recently, it has been generalized to bounded lattices by Kara\c{c}al et al.~%
\cite{KIM2015}, and extensively investigated. In particular, Kara\c{c}al et
al.~\cite{KIM2015} proved the existence of nullnorms with the zero element $a$
for arbitrary element $a\in L\backslash \{0,1\}$ with underlying t-norms and
t-conorms on an arbitrary bounded lattice $L$. \c{C}ayl{\i } and Kara\c{c}al
\cite{CK2018} showed that there exists a unique idempotent nullnorm on an
arbitrary distributive bounded lattice and proved that an idempotent
nullnorm on a bounded lattice need not always exists. Ertu\v{g}rul~\cite{E2018}
not only obtained two new methods to construct nullnorms on an arbitrary
bounded lattice, but also introduced and investigated an equivalent relation based
on a class of nullnorms. Xie~\cite{Xie2018} extended type-1 proper nullnorms and
proper uninorms to fuzzy truth values and introduced the notions of type-2
nullnorms and uninorms. \c{C}ayl\i~\cite{Ca2020} discussed the characterization of idempotent
nullnorms on bounded lattices and researched some of their main properties.
Moreover, it was presented some construction approaches for nullnorms, in particular
idempotent nullnorms, on bounded lattices with the indicated zero element
under some additional constraints. In~\cite{Ca2020-2}, two new methods were
introduced to obtain idempotent nullnorms on bounded lattices
with a zero element $a$ under the additional assumption that all elements
incomparable with $a$ are comparable with each other. In~\cite{Ca2020-3},
four methods were proposed to construct nullnorms on bounded lattices
derived from triangular norms on $\left[a, 1\right]^{2}$ and triangular
conorms on $\left[0, a\right]^{2}$, where some sufficient and necessary
conditions on theirs zero element $a$ are required. As a consequence of
these methods, some new types idempotent nullnorms on bounded lattices were obtained.
Sun and Liu~\cite{SL2020} characterized a class of nullnorms on bounded lattices
and represented them by triangular norms and triangular conorms.

\c{C}ayl{\i} and Kara\c{c}al obtained a result (see \cite[Theorem 3]{CK2018})
stating that there exists a unique idempotent nullnorm on bounded lattice
$L$ with the zero element $a$ provided that the bounded lattice $L$ contains
only one element incomparable with $a$. This paper considers when does the
idempotent nullnorm on bounded lattice $L$ with the zero element $a$ exist,
under the case that the bounded lattice $L$ contains two distinct element
incomparable with $a$. By using this equivalent characterization, the bounded
lattices which do not admit any idempotent nullnorm on themselves with the zero
element $a$ are obtained. Meanwhile, some basic properties for the bounded
lattice containing two distinct element incomparable with $a$ are obtained.

\section{Preliminaries}

This section provides some basic concepts based on bounded lattices and some
results related to them.

A \textit{lattice} \cite{Bir1967} is a partially ordered set $(L,\leq )$
satisfying that each two elements $x,y\in L$ have a greatest lower bound,
called \textit{infimum} and denoted as $x\wedge y$, as well as a smallest
upper bound, called \textit{supremum} and denoted by $x\vee y$. A lattice is
called \textit{bounded} if it has a top element and a bottom element,
written as $1$ and $0$, respectively. Let $(L,\leq ,0,1)$ denote a bounded
lattice with top element $1$ and bottom element $0$ throughout this paper.
For $x,y\in L$, the symbol $x<y$ means that $x\leq y$ and $x\neq y$. The
elements $x$ and $y$ in $L$ are \textit{comparable} if $x\leq y$ or $y\leq x$,
in this case, we use the notation $x\nparallel y$. Otherwise, $x$ and $y$
are called \textit{incomparable} if $x\nleq y$ and $y\nleq x$, in this case,
we use the notation $x\Vert y$. The set of all elements of $L$ that are
incomparable with $a$ is denoted by $I_{a}$, i.e., $I_{a}=\{x\in L:x\Vert a\}$.

%

\begin{definition}
\cite{Bir1967} Let $(L, \leq, 0, 1)$ be a bounded lattice and $a, b\in L$
with $a\leq b$. The subinterval $[a, b]$ is defined as
\begin{equation*}
[a, b]=\{x\in L: a\leq x\leq b\}.
\end{equation*}
Other subintervals such as $[a, b)$, $(a, b]$, and $(a, b)$ can be defined
similarly. Obviously, $([a, b], \leq)$ is a bounded lattice with top element
$b$ and bottom element $a$.
\end{definition}

\begin{definition}
\cite{KIM2015} Let $(L, \leq, 0, 1)$ be a bounded lattice. A commutative,
associative, non-decreasing in each variable function $V: L^2\to L$ is
called a \textit{nullnorm} on $L$ if there exists an element $a\in L$, which
is called a \textit{zero element} for $V$, such that $V(x, 0)=x$ for all $%
x\leq a$ and $V(x, 1)=x$ for all $x\geq a$.
\end{definition}

Clearly, $V(x, a)=a$ for all $x\in L$.

\begin{definition}
{\cite{Ca2020-2}} Let $(L,\leq ,0,1)$ be a bounded lattice, $a\in
L\backslash \{0,1\}$, and $V$ be a nullnorm on $L$ with the zero element $a$%
. $V$ is called an \textit{idempotent nullnorm} on $L$ if $V(x,x)=x$ for all
$x\in L$.
\end{definition}

The following basic properties on the idempotent nullnorms are obtained by
\c{C}ayl{\i}~\cite{Ca2020}, showing that an idempotent nullnorm on $L$ with
the zero element $a$ is completely determined by its values on $I_{a}\times
I_{a}$.

\begin{lemma}{\rm \cite[Propositions 3, 6 and 7]{Ca2020}}\label{stru-lemma}
Let $(L, \leq, 0, 1)$ be a bounded lattice, $a\in L\backslash \{0, 1\}$, and $V$ be
an idempotent nullnorm on $L$ with the zero element $a$. Then, the following
statements hold:
\begin{enumerate}[{\rm (i)}]
  \item $V(x, y)=x\vee y$ for all $(x, y)\in [0, a]^{2}$;
  \item $V(x, y)=x\wedge y$ for all $(x, y)\in [a, 1]^{2}$;
  \item $V(x, y)=x\vee (y\wedge a)$ for all $(x, y)\in [0, a]\times I_a$;
  \item $V(x, y)=y\vee (x\wedge a)$ for all $(x, y)\in I_a\times [0, a]$;
  \item $V(x, y)=x\wedge (y\vee a)$ for all $(x, y)\in [a, 1]\times I_a$;
  \item $V(x, y)=y\wedge (x\vee a)$ for all $(x, y)\in I_a\times [a, 1]$;
  \item $V(x, y)=a$ for all $(x, y)\in \left([0, a]\times [a, 1]\right) \cup \left([a, 1]\times [0, a]\right)$;
  \item\label{viii} $V(x, y)=(x\wedge a)\vee (y\wedge a)$ or $V(x, y)=(x\vee a)\wedge (y\vee a)$ or
$V(x, y)\in I_a$ for all $x, y\in I_{a}$.
\end{enumerate}
\end{lemma}

\section{Characterization of idempotent nullnorms}

This section studies idempotent nullnorms with the zero element $a$ on a
bounded lattice containing two distinct elements incomparable with $a$ and
obtains an equivalent characterization for its existence. Moreover, some
basic properties for the bounded lattice containing two distinct element
incomparable with $a$ are obtained.

\begin{lemma}
\label{Ia-Operation} Let $(L, \leq, 0, 1)$ be a bounded lattice, $a\in
L\backslash \{0, 1\}$, and $V$ be an idempotent nullnorm on $L$ with the
zero element $a$. For any $x, y\in I_{a}$, if $(x\wedge a)\vee (y\wedge a)<a$
and $(x\vee a)\wedge (y\vee a)>a$, then $V(x, y)\in I_a$.
\end{lemma}

\begin{proof}
For any $x, y\in I_{a}$, consider the following two cases:
\begin{itemize}
  \item[Case 1.] $V(x, y)=(x\wedge a)\vee (y\wedge a)<a$. Applying Lemma~\ref{stru-lemma}, it follows that
  $V(1, V(x, y))=a$ and $V(V(1, x), y)=V(x\vee a, y)=(x\vee a)\wedge (y\vee a)>a$, which contracts with the associativity of $V$;
  \item[Case 2.] $V(x, y)=(x\vee a)\wedge (y\vee a)>a$. Applying Lemma~\ref{stru-lemma}, it follows that
  $V(0, V(x, y))=a$ and $V(V(0, x), y)=V(x\wedge a, y)=(x\wedge a)\vee (y\wedge a)<a$, which contracts with the associativity of $V$.
\end{itemize}
This, together with Lemma \ref{stru-lemma}--(\ref{viii}), implies that $V(x, y)\in I_{a}$.
\end{proof}

\begin{theorem}
\label{Wu-Operation} Let $(L, \leq, 0, 1)$ be a bounded lattice, $a\in
L\backslash \{0, 1\}$, and $I_{a}=\{p, q\}$ with $p\neq q$. Then, there
exists an idempotent nullnorm on $L$ with the zero element $a$ if, and only
if, one of the following statements holds:

\begin{enumerate}[{\rm (i)}]
  \item\label{i} $(p\wedge a)\vee (q\wedge a)=a$;
  \item $(p\vee a)\wedge (q\vee a)=a$;
  \item $p\vee a\leq q\vee a$ and $q\wedge a\leq p\wedge a$;
  \item\label{iv} $p\wedge a\leq q\wedge a$ and $q\vee a\leq p\vee a$.
\end{enumerate}
\end{theorem}

\begin{proof}
{\it Necessity.}

Let $V$ be an idempotent nullnorm on $L$ with the zero element $a$ and suppose on the
contrary that none of statements (\ref{i})--(\ref{iv}) holds. This implies that the following hold:
\begin{itemize}
  \item[(i$^{\prime}$)] $(p\wedge a)\vee (q\wedge a)<a$;
  \item[(ii$^{\prime}$)] $(p\vee a)\wedge (q\vee a)>a$;
  \item[(iii$^{\prime}$)] $p\vee a\nleq q\vee a$ or $q\wedge a\nleq p\wedge a$;
  \item[(iv$^{\prime}$)] $p\wedge a\nleq q\wedge a$ or $q\vee a\nleq p\vee a$.
\end{itemize}
Applying Lemma~\ref{Ia-Operation}, (i$^{\prime}$) and (ii$^{\prime}$), it follows that $V(p, q)=V(q, p)\in I_{a}$. Noting that
$I_{a}=\{p, q\}$, it suffices to consider the following two cases:
\begin{enumerate}[(1)]
  \item If $V(p, q)=V(q, p)=p$, then
  $$
  V(0, V(p, q))=V(0, p)=p\wedge a, \quad V(V(0, p), q)=V(p\wedge a, q)=(p\wedge a)\vee (q\wedge a),
  $$
  and
  $$
  V(1, V(p, q))=V(1, p)=p\vee a, \quad V(V(1, p), q)=V(p\vee a, q)=(p\vee a)\wedge (q\vee a).
  $$
  This, together with the associativity of $V$, implies that
  $$
  p\wedge a=(p\wedge a)\vee (q\wedge a),
  $$
  and
  $$
  p\vee a=(p\vee a)\wedge (q\vee a),
  $$
  and thus $q\wedge a\leq p\wedge a$ and $p\vee a\leq q\vee a$. This contracts with (iii$^{\prime}$).

  \item If $V(p, q)=V(q, p)=q$, similarly to the above proof, it can be verified that $p\wedge a\leq q\wedge a$ and $q\vee a\leq p\vee a$,
  which contracts with (iv$^{\prime}$).
\end{enumerate}

{\it Sufficiency.}

\begin{enumerate}[{\rm (i)}]
  \item $(p\wedge a)\vee (q\wedge a)=a$. From \cite[Theorem~4]{Ca2020}, it follows that
  the binary operation $V_1: L^{2}\to L$ defined by
  $$
  V_{1}(x, y)=
  \begin{cases}
  x\vee y, & (x, y)\in [0, a]^{2}, \\
  x\wedge y, & (x, y)\in [a, 1]^{2}, \\
  a, & (x, y)\in \left([0, a]\times [a, 1]\right) \cup \left([a, 1]\times [0, a]\right), \\
  x\vee (y\wedge a), & (x, y)\in [0, a]\times I_{a}, \\
  y\vee (x\wedge a), & (x, y)\in I_{a}\times [0, a], \\
  x\wedge (y\vee a), & (x, y)\in [a, 1]\times I_{a}, \\
  y\wedge (x\vee a), & (x, y)\in I_{a}\times [a, 1], \\
  x, & x, y\in I_{a} \text{ and } x=y, \\
  (x\vee a)\wedge (y\vee a), & (x, y)\in \{(p, q), (q, p)\},
  \end{cases}
  $$
  is an idempotent nullnorm on $L$ with the zero element $a$.
  \item $(p\vee a)\wedge (q\vee a)=a$. From \cite[Theorem~5]{Ca2020}, it follows that
  the binary operation $V_2: L^{2}\to L$ defined by
  $$
  V_{2}(x, y)=
  \begin{cases}
  x\vee y, & (x, y)\in [0, a]^{2}, \\
  x\wedge y, & (x, y)\in [a, 1]^{2}, \\
  a, & (x, y)\in \left([0, a]\times [a, 1]\right) \cup \left([a, 1]\times [0, a]\right), \\
  x\vee (y\wedge a), & (x, y)\in [0, a]\times I_{a}, \\
  y\vee (x\wedge a), & (x, y)\in I_{a}\times [0, a], \\
  x\wedge (y\vee a), & (x, y)\in [a, 1]\times I_{a}, \\
  y\wedge (x\vee a), & (x, y)\in I_{a}\times [a, 1], \\
  x, & x, y\in I_{a} \text{ and } x=y, \\
  (x\wedge a)\vee (y\wedge a), & (x, y)\in \{(p, q), (q, p)\},
  \end{cases}
  $$
  is an idempotent nullnorm on $L$ with the zero element $a$.
  \item $p\vee a\leq q\vee a$ and $q\wedge a\leq p\wedge a$. We claim that the binary operation $V_3: L^{2}\to L$ defined by
  \begin{equation}\label{eq-1}
  V_{3}(x, y)=
  \begin{cases}
  x\vee y, & (x, y)\in [0, a]^{2}, \\
  x\wedge y, & (x, y)\in [a, 1]^{2}, \\
  a, & (x, y)\in \left([0, a]\times [a, 1]\right) \cup \left([a, 1]\times [0, a]\right), \\
  x\vee (y\wedge a), & (x, y)\in [0, a]\times I_{a}, \\
  y\vee (x\wedge a), & (x, y)\in I_{a}\times [0, a], \\
  x\wedge (y\vee a), & (x, y)\in [a, 1]\times I_{a}, \\
  y\wedge (x\vee a), & (x, y)\in I_{a}\times [a, 1], \\
  x, & x, y\in I_{a} \text{ and } x=y, \\
  p, & (x, y)\in \{(p, q), (q, p)\},
  \end{cases}
  \end{equation}
  is an idempotent nullnorm on $L$ with the zero element $a$.

  First, it can be verified that $V_{3}$ is a commutative binary operation with the zero element $a$.
It remains to check the monotonicity and associativity of $V_{3}$.

\begin{itemize}
  \item[iii-1)] {\it Monotonicity.} For any $x,y \in L$ with $x \leq y$, it holds that $V_{3}(x,z)\leq V_{3}(y,z)$ for all $z\in L$.
Consider the following cases:
\begin{itemize}
  \item[1.] $x\in [0, a)$. It is easy to see that
  \begin{equation}\label{eq-1.1}
  V_{3}(x, z)=
  \begin{cases}
  x\vee z, & z\in [0, a), \\
  a, & z\in [a, 1], \\
  x\vee (z\wedge a), & z\in I_{a}.
  \end{cases}
  \end{equation}
  \begin{itemize}
    \item[1.1.] If $y\in [0, a)$, applying \eqref{eq-1.1}, it is clear that $V_{3}(x, z)\leq V_{3}(y, z)$ holds for all $z\in L$.
    \item[1.2.] If $y\in [a, 1]$, applying \eqref{eq-1} and \eqref{eq-1.1}, it can be verified that
    \begin{align}
    V_{3}(y, z)&=
    \begin{cases}\label{eq-1.2}
    a, & z\in [0, a), \\
    y\wedge z, & z\in [a, 1], \\
    y\wedge (z\vee a), & z\in I_{a},
    \end{cases}\\
    &\geq
    \begin{cases}\nonumber
    x\vee z, & z\in [0, a), \\
    a, & z\in [a, 1], \\
    a, & z\in I_{a},
    \end{cases}\\
    & \geq V_{3}(x, z).\nonumber
    \end{align}
    \item[1.3.] If $y=p$, it follows from $x\leq y=p$, $0\leq x<a$, and $q\wedge a\leq p\wedge a$
    that $p\wedge a\geq x\wedge a= x$, $x\vee (p\wedge a)\leq p\vee (p\wedge a) \leq p$, and $x\vee (q\wedge a)\leq p\vee (q\wedge a)\leq p\vee (p\wedge a)=p$.
    These, together with \eqref{eq-1} and \eqref{eq-1.1}, imply that
    \begin{align*}
    V_{3}(y, z)=V_{3}(p, z)&=
    \begin{cases}
    z\vee (p\wedge a), & z\in [0, a), \\
    z\wedge (p\vee a), & z\in [a, 1], \\
    p, & z=p, \\
    p, & z=q,
    \end{cases}\\
    & \geq
    \begin{cases}
    z\vee x, & z\in [0, a), \\
    a, & z\in [a, 1], \\
    x\vee (p\wedge a), & z=p, \\
    x\vee (q\wedge a), & z=q,
    \end{cases}\\
    &= V_{3}(x, z).
    \end{align*}
    \item[1.4.] If $y=q$, it follows from $x\leq y=q$, $0\leq x<a$, and $q\wedge a\leq p\wedge a$
    that $q\wedge a\geq x\wedge a=x$, $x\vee (p\wedge a)\leq (q\wedge a)\vee (p\wedge a)=p\wedge a\leq p$,
    and $x\vee (q\wedge a)\leq q \vee (q\wedge a)=q$.
    These, together with \eqref{eq-1} and \eqref{eq-1.1}, imply that
    \begin{align*}
    V_{3}(y, z)=V_{3}(q, z)&=
    \begin{cases}
    z\vee (q\wedge a), & z\in [0, a), \\
    z\wedge (q\vee a), & z\in [a, 1], \\
    q, & z=q, \\
    p, & z=p,
    \end{cases}\\
    & \geq
    \begin{cases}
    z\vee x, & z\in [0, a), \\
    a, & z\in [a, 1], \\
    x\vee (q\wedge a), & z=q, \\
    x\vee (p\wedge a), & z=p,
    \end{cases}\\
    &= V_{3}(x, z).
    \end{align*}
  \end{itemize}

  \item[2.] $x\in [a, 1]$. From $x\leq y$, it is clear that $y\in [a, 1]$. Applying \eqref{eq-1.2},
  it can be verified that $V_{3}(x, z)\leq V_{3}(y, z)$ holds for all $z\in L$.

  \item[3.] $x\in I_{a}$. From $x\leq y$, it follows that $y\in [a, 1]\cup I_{a}$. Applying \eqref{eq-1},
  it can be verified that
  \begin{equation}\label{eq-2}
  V_{3}(x, z)=
    \begin{cases}
    z\vee (x\wedge a), & z\in [0, a), \\
    z\wedge (x\vee a), & z\in [a, 1], \\
    x, & z=x, \\
    p, & z\in I_{a}\backslash \{x\}.
    \end{cases}
  \end{equation}
  \begin{itemize}
    \item[3.1.] If $y\in I_{a}$, consider the following cases:
    \begin{itemize}
      \item[3.1.1.] $x=y$. It is clear that $V_{3}(x, z)\leq V_{3}(y, z)$ holds for all $z\in L$.
      \item[3.1.2.] $p=x<y=q$. Applying~\eqref{eq-2} yields that
    $$
    V_{3}(x, z)=V_{3}(p, z)=
    \begin{cases}
    z\vee (p\wedge a), & z\in [0, a), \\
    z\wedge (p\vee a), & z\in [a, 1], \\
    p, & z=p, \\
    p, & z=q,
    \end{cases}
    $$
    and
    $$
    V_{3}(y, z)=V_{3}(q, z)=
    \begin{cases}
    z\vee (q\wedge a), & z\in [0, a), \\
    z\wedge (q\vee a), & z\in [a, 1], \\
    p, & z=p,\\
    q, & z=q,
    \end{cases}
    $$
    implying that $V_{3}(x, z)\leq V_{3}(y, z)$ holds for all $z\in L$.

    \item[3.1.3.] $q=x<y=p$. Applying~\eqref{eq-2} yields that
    $$
    V_{3}(x, z)=V_{3}(q, z)=
    \begin{cases}
    z\vee (q\wedge a), & z\in [0, a), \\
    z\wedge (q\vee a), & z\in [a, 1], \\
    p, & z=p,\\
    q, & z=q.
    \end{cases}
    $$
    and
    $$
    V_{3}(y, z)=V_{3}(p, z)=
    \begin{cases}
    z\vee (p\wedge a), & z\in [0, a), \\
    z\wedge (p\vee a), & z\in [a, 1], \\
    p, & z=p, \\
    p, & z=q,
    \end{cases}
    $$
    implying that $V_{3}(x, z)\leq V_{3}(y, z)$ holds for all $z\in L$.
    \end{itemize}

    \item[3.2.] If $y\in [a, 1]$, consider the following cases:
    \begin{itemize}
      \item[3.2.1.] $x=p$. From $p=x\leq y$ and $p\vee a\leq q\vee a$, it follows that
      $y\wedge (p\vee a)\geq x\wedge (p\vee a)=p\wedge (p\vee a)=p$
      and $y\wedge (q\vee a)\geq x\wedge (q\vee a)\geq x\wedge (p\vee a)=p\wedge (p\vee a)=p$. These, together with \eqref{eq-2} and \eqref{eq-1.2}, imply that
    \begin{align*}
    V_{3}(x, z)=V_{3}(p, z)&=
    \begin{cases}
    z\vee (p\wedge a), & z\in [0, a), \\
    z\wedge (p\vee a), & z\in [a, 1], \\
    p, & z=p, \\
    p, & z=q,
    \end{cases}\\
    & \leq
    \begin{cases}
    a, & z\in [0, a), \\
    z\wedge (y\vee a), & z\in [a, 1], \\
    y\wedge (p\vee a), & z=p, \\
    y\wedge (q\vee a), & z=q,
    \end{cases}\\
    &=V_{3}(y, z).
    \end{align*}
      \item[3.2.2.] $x=q$. From $q=x\leq y$, $a\leq y\leq 1$, and $p\vee a\leq q\vee a$,
      it follows that $y\wedge (p\vee a)=(y\vee a)\wedge (p\vee a)\geq (q\vee a)\wedge (p\vee a)
      =p\vee a\geq p$ and $y\wedge (q\vee a)\geq q \wedge (q\vee a)=q$. These, together with \eqref{eq-2} and \eqref{eq-1.2}, imply that
    \begin{align*}
    V_{3}(x, z)=V_{3}(q, z)&=
    \begin{cases}
    z\vee (q\wedge a), & z\in [0, a), \\
    z\wedge (q\vee a), & z\in [a, 1], \\
    q, & z=q, \\
    p, & z=p,
    \end{cases}\\
    & \leq
    \begin{cases}
    a, & z\in [0, a), \\
    z\wedge (y\vee a), & z\in [a, 1], \\
    y\wedge (q\vee a), & z=q, \\
    y\wedge (p\vee a), & z=p,
    \end{cases}\\
    &=V_{3}(y, z).
    \end{align*}
    \end{itemize}
  \end{itemize}
\end{itemize}
  \item[iii-2)] {\it Associativity.} For any $x, y, z\in L$, it holds that $V_{3}(x,V_{3}(y,z))=V_{3}(V_{3}(x,y),z)$.
  Consider the following cases:

\begin{itemize}

\item[1.] $x\in [0, a]$.
\begin{itemize}
\item[1.1.] $y\in [0, a]$.
\begin{itemize}
\item[1.1.1.] If $z\in [0, a]$, this holds trivially.

\item[1.1.2.] If $z\in [a, 1]$, then $V_{3}(x, V_{3}(y, z))=V_{3}(x, a)=a=V_{3}(x\vee y, z)
=V_{3}(V_{3}(x, y), z)$.

\item[1.1.3.] If $z\in I_{a}$, then $V_{3}(x, V_{3}(y, z))=V_{3}(x, y\vee (z\wedge a))
=x\vee (y \vee (z\wedge a))=(x\vee y)\vee (z\wedge a)=V_{3}(x, y)\vee (z\wedge a)=
V_{3}(V_{3}(x, y), z)$.
\end{itemize}

\item[1.2.] $y\in [a, 1]$.
\begin{itemize}
\item[1.2.1.] If $z\in [0, a]$, then $V_{3}(x, V_{3}(y, z))=V_{3}(x, a)=a=V_{3}(a, z)=V_{3}(V_{3}(x, y), z)$.

\item[1.2.2.] If $z\in [a, 1]$, then $V_{3}(x, V_{3}(y, z))=V_{3}(x, y\wedge z)=a=
V_{3}(a, z)=V_{3}(V_{3}(x, y), z)$.

\item[1.2.3.] If $z\in I_{a}$, then $V_{3}(x, V_{3}(y, z))=V_{3}(x, y\wedge (z\vee a))=a
=V_{3}(a, z)=V_{3}(V_{3}(x, y), z)$.
\end{itemize}

\item[1.3.] $y\in I_{a}$.
\begin{itemize}
\item[1.3.1.] If $z\in [0, a]$, then $V_{3}(x, V_{3}(y, z))=V_{3}(x, (y\wedge a)\vee z)
=x\vee (y\wedge a)\vee z=V_{3}(x, y)\vee z
=V_{3}(V_{3}(x, y), z)$.

\item[1.3.2.] If $z\in [a, 1]$, then $V_{3}(x, V_{3}(y, z))=V_{3}(x, (y\vee a)\wedge z)=a
=V_{3}(x\vee (y\wedge a), z)=V_{3}(V_{3}(x, y), z)$.

\item[1.3.3.] If $z\in I_{a}$, consider the following cases:

1.3.3.1. $y=z$. It can be verified that
$$
V_{3}(x, V_{3}(y, z))=V_{3}(x, y)=x\vee (y\wedge a),
$$
and
$$
V_{3}(V_{3}(x, y), z)=V_{3}(x\vee (y\wedge a), y)=x\vee (y\wedge a)\vee (y\wedge a)=x\vee (y\wedge a),
$$
implying that $V_{3}(x, V_{3}(y, z))=V_{3}(V_{3}(x, y), z)$.

1.3.3.2. $y\neq z$. Applying $q\wedge a\leq p\wedge a$, it can be verified that
$$
V_{3}(x, V_{3}(y, z))=V_{3}(x, p)=x\vee (p\wedge a),
$$
and
$$
V_{3}(V_{3}(x, y), z)=V_{3}(x\vee (y\wedge a), z)=x\vee (y\wedge a)\vee (z\wedge a)=x\vee (p\wedge a),
$$
implying that $V_{3}(x, V_{3}(y, z))=V_{3}(V_{3}(x, y), z)$.

\end{itemize}
\end{itemize}

\item[2.] $x\in [a, 1]$.
\begin{itemize}
\item[2.1.] $y\in [0, a]$.

\begin{itemize}
\item[2.1.1.] If $z\in [0, a]$, then
\begin{align*}
V_{3}(x, V_{3}(y, z))&=V_{3}(V_{3}(z, y), x)~~ (\text{commutativity of } V_{3})\\
&= V_{3}(z, V_{3}(y, x))~~ (\text{by 1.1.2})\\
&=V_{3}(V_{3}(x, y), z)~~ (\text{commutativity of } V_{3}).
\end{align*}

\item[2.1.2.] If $z\in [a, 1]$, then $V_{3}(x, V_{3}(y, z))=V_{3}(x, a)=a=
V_{3}(a, z)=V_{3}(V_{3}(x, y), z)$.

\item[2.1.3.] If $z\in I_{a}$, then $V_{3}(x, V_{3}(y, z))=V_{3}(x, y\vee (z\wedge a))=
a=V_{3}(a, z)=V_{3}(V_{3}(x, y), z)$.
\end{itemize}

\item[2.2.] $y\in [a, 1]$.
\begin{itemize}
\item[2.2.1.] If $z\in [0, a)$, then
\begin{align*}
V_{3}(x, V_{3}(y, z))&=V_{3}(V_{3}(z, y), x)~~ (\text{commutativity of } V_{3})\\
&=V_{3}(z, V_{3}(y, x))~~ (\text{by 1.2.2})\\
&=V_{3}(V_{3}(x, y), z)~~ (\text{commutativity of } V_{3}).
\end{align*}

\item[2.2.2.] If $z\in [a, 1]$, this holds trivially.

\item[2.2.3.] If $z\in I_{a}$, then
$V_{3}(x, V_{3}(y, z))=V_{3}(x, y\wedge (z\vee a))=x\wedge y \wedge (z\vee a)
=V_{3}(x, y)\wedge (z\vee a)=
V_{3}(V_{3}(x, y), z)$.
\end{itemize}

\item[2.3.] $y\in I_{a}$.
\begin{itemize}

\item[2.3.1.] If $z\in [0, a]$, then
\begin{align*}
V_{3}(x, V_{3}(y, z))&=V_{3}(V_{3}(z, y), x) ~~ (\text{commutativity of } V_{3})\\
&=V_{3}(z, V_{3}(y, x))~~ (\text{by 1.3.2})\\
&=V_{3}(V_{3}(x, y), z)~~ (\text{commutativity of } V_{3}).
\end{align*}

\item[2.3.2.] If $z\in [a, 1]$, then
$V_{3}(x, V_{3}(y, z))=V_{3}(x, (y\vee a)\wedge z)=x\wedge (y\vee a)\wedge z
=V_{3}(x, y)\wedge z=
V_{3}(V_{3}(x, y), z)$.

\item[2.3.3.] If $z\in I_{a}$, consider the following cases:

2.3.3.1. $y=z$. It can be verified that
$$
V_{3}(x, V_{3}(y, z))=V_{3}(x, y)=x\wedge (y\vee a),
$$
and
$$
V_{3}(V_{3}(x, y), z)=V_{3}(x\wedge (y\vee a), y)=x\wedge (y\vee a)\wedge (y\vee a)=x\wedge (y\vee a),
$$
implying that $V_{3}(x, V_{3}(y, z))=V_{3}(V_{3}(x, y), z)$.

2.3.3.2. $y\neq z$. Applying $p\vee a\leq q\vee a$, it can be verified that
$$
V_{3}(x, V_{3}(y, z))=V_{3}(x, p)=x\wedge (p\vee a),
$$
and
$$
V_{3}(V_{3}(x, y), z)=V_{3}(x\wedge (y\vee a), z)=x\wedge (y\vee a)\wedge (z\vee a)
=x\wedge (p\vee a),
$$
implying that $V_{3}(x, V_{3}(y, z))=V_{3}(V_{3}(x, y), z)$.

\end{itemize}
\end{itemize}

\item[3.] $x\in I_{a}$.
\begin{itemize}
\item[3.1.] $y\in [0, a]$.
\begin{itemize}
\item[3.1.1.] If $z\in [0, a]$, then
\begin{align*}
V_{3}(x, V_{3}(y, z))&=V_{3}(V_{3}(z, y), x) ~~ (\text{commutativity of } V_{3})\\
&=V_{3}(z, V_{3}(y, x))~~ (\text{by 1.1.3})\\
&=V_{3}(V_{3}(x, y), z)~~ (\text{commutativity of } V_{3}).
\end{align*}

\item[3.1.2.] If $z\in [a, 1]$, then
\begin{align*}
V_{3}(x, V_{3}(y, z))&=V_{3}(V_{3}(z, y), x) ~~ (\text{commutativity of } V_{3})\\
&=V_{3}(z, V_{3}(y, x))~~ (\text{by 2.1.3})\\
&=V_{3}(V_{3}(x, y), z)~~ (\text{commutativity of } V_{3}).
\end{align*}

\item[3.1.3.] If $z\in I_{a}$, then $V_{3}(x, V_{3}(y, z))=V_{3}(x, y\vee (z\wedge a))
=(x\wedge a)\vee y\vee (z\wedge a)=V_{3}(x, y)\vee (z\wedge a)
=V_{3}(V_{3}(x, y), z)$.
\end{itemize}

\item[3.2.] $y\in [a, 1]$.
\begin{itemize}
\item[3.2.1.] If $z\in [0, a]$, then
\begin{align*}
V_{3}(x, V_{3}(y, z))&=V_{3}(V_{3}(z, y), x)~~ (\text{commutativity of } V_{3})\\
&= V_{3}(z, V_{3}(y, x))~~ (\text{by 1.2.3})\\
&=V_{3}(V_{3}(x, y), z)~~ (\text{commutativity of } V_{3}).
\end{align*}

\item[3.2.2.] If $z\in [a, 1]$, then
\begin{align*}
V_{3}(x, V_{3}(y, z))&=V_{3}(V_{3}(z, y), x)~~ (\text{commutativity of } V_{3})\\
&=V_{3}(z, V_{3}(y, x))~~ (\text{by 2.2.3})\\
&=V_{3}(V_{3}(x, y), z)~~ (\text{commutativity of } V_{3}).
\end{align*}

\item[3.2.3.] If $z\in I_{a}$, then
$V_{3}(x, V_{3}(y, z))=V_{3}(x, y\wedge (z\vee a))=
(x\vee a)\wedge y\wedge (z\vee a)
=V_{3}(x, y)\wedge (z\vee a)=
V_{3}(V_{3}(x, y), z)$.
\end{itemize}

\item[3.3.] $y\in I_{a}$.
\begin{itemize}
\item[3.3.1.] If $z\in [0, a]$, then
\begin{align*}
V_{3}(x, V_{3}(y, z))&=V_{3}(V_{3}(z, y), x)~~ (\text{commutativity of } V_{3})\\
&=V_{3}(z, V_{3}(y, x))~~ (\text{by 1.3.3})\\
&=V_{3}(V_{3}(x, y), z)~~ (\text{commutativity of } V_{3}).
\end{align*}

\item[3.3.2.] If $z\in [a, 1]$, then
\begin{align*}
V_{3}(x, V_{3}(y, z))&=V_{3}(V_{3}(z, y), x)~~ (\text{commutativity of } V_{3})\\
&=V_{3}(z, V_{3}(y, x))~~ (\text{by 2.3.3})\\
&=V_{3}(V_{3}(x, y), z)~~ (\text{commutativity of } V_{3}).
\end{align*}

\item[3.3.3.] If $z\in I_{a}$, it is not difficult to check that
$$
V_{3}(x, V_{3}(y, z))=
\begin{cases}
q, & x=y=z=q, \\
p, & \text{otherwise},
\end{cases}
$$
and
$$
V_{3}(V_{3}(x, y), z)=
\begin{cases}
q, & x=y=z=q, \\
p, & \text{otherwise}.
\end{cases}
$$

\end{itemize}
\end{itemize}
\end{itemize}
\end{itemize}
 \item $p\wedge a\leq q\wedge a$ and $q\vee a\leq p\vee a$. Similarly to the proof of (iii), it can be verified that
 the binary operation $V_4: L^{2}\to L$ defined by
  \begin{equation}\label{eq-2.5}
  V_{4}(x, y)=
  \begin{cases}
  x\vee y, & (x, y)\in [0, a]^{2}, \\
  x\wedge y, & (x, y)\in [a, 1]^{2}, \\
  a, & (x, y)\in \left([0, a]\times [a, 1]\right) \cup \left([a, 1]\times [0, a]\right), \\
  x\vee (y\wedge a), & (x, y)\in [0, a]\times I_{a}, \\
  y\vee (x\wedge a), & (x, y)\in I_{a}\times [0, a], \\
  x\wedge (y\vee a), & (x, y)\in [a, 1]\times I_{a}, \\
  y\wedge (x\vee a), & (x, y)\in I_{a}\times [a, 1], \\
  x, & x, y\in I_{a} \text{ and } x=y, \\
  q, & (x, y)\in \{(p, q), (q, p)\},
  \end{cases}
  \end{equation}
  is an idempotent nullnorm on $L$ with the zero element $a$.
\end{enumerate}
\end{proof}

\begin{corollary}
\label{Nonexistence-Cor} If a bounded lattice $L$ contains a sublattice
which is isomorphic to a sublattice characterized by Hasse diagram in Fig.~%
\ref{FIG:1}, then there is no idempotent nullnorm on $L$ with the zero
element $a$.
\begin{figure}[h]
\begin{center}
\scalebox{0.7}{\includegraphics{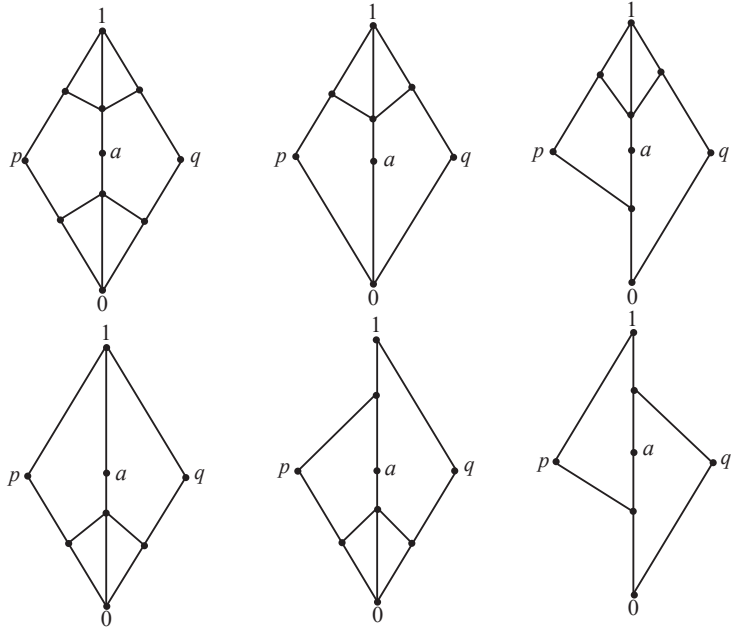}}
\end{center}
\caption{The lattices in Corollary~\protect\ref{Nonexistence-Cor}}
\label{FIG:1}
\end{figure}
\end{corollary}

\begin{proof}
It follows directly from Theorem~\ref{Wu-Operation}.
\end{proof}

\begin{corollary}
Let $(L, \leq, 0, 1)$ be a bounded lattice, $a\in L\backslash \{0, 1\}$, and
$I_{a}=\{p, q\}$ such that $p\neq q$ and $p\nparallel q$. Then, we have $%
p\wedge a=q\wedge a$ or $p\vee a=q\vee a$.
\end{corollary}

\begin{proof}
Without loss of generality, assume that $p\leq q$. This implies that $(p\wedge a) \vee (q\wedge a)=q\wedge a<a$ and
$(p\vee a)\wedge (q\vee a)=p\vee a>a$. From \cite[Theorem 3]{Ca2020-2}, it follows that
there exists an idempotent nullnorm on $L$ with the zero element $a$. These, together with Theorem~\ref{Wu-Operation},
imply that
\begin{enumerate}[(a)]
  \item $p\vee a\leq q\vee a$ and $q\wedge a\leq p\wedge a$; or
  \item $p\wedge a\leq q\wedge a$ and $q\vee a\leq p\vee a$.
\end{enumerate}
If $p\wedge a=q\wedge a$, the theorem is completed. Otherwise, if $p\wedge a< q\wedge a$, then (b) is satisfied,
implying that $q\vee a\leq p\vee a$. This, together with $p\leq q$, implies that $p\vee a=q\vee a$.
\end{proof}

\begin{corollary}
\label{cor1-1} Let $\left( L,\leq ,0,1\right) $ be a bounded lattice, $a\in
L\backslash \left\{ 0,1\right\} $ and $I_{a}=\left\{ p,q\right\} $ such that
$p\neq q$ and $p\nparallel q$. Then, the binary operation $V_{5}:L\times
L\rightarrow L$ defined by
\begin{equation}
V_{5}\left( x,y\right) =%
\begin{cases}
x\vee y, & (x,y)\in \lbrack 0,a]^{2}, \\
x\wedge y, & (x,y)\in \lbrack a,1]^{2}, \\
a, & (x,y)\in \left( \lbrack 0,a]\times \lbrack a,1]\right) \cup \left(
\lbrack a,1]\times \lbrack 0,a]\right) , \\
x\vee (y\wedge a), & (x,y)\in \lbrack 0,a]\times I_{a}, \\
y\vee (x\wedge a), & (x,y)\in I_{a}\times \lbrack 0,a], \\
x\wedge (y\vee a), & (x,y)\in \lbrack a,1]\times I_{a}, \\
y\wedge (x\vee a), & (x,y)\in I_{a}\times \lbrack a,1], \\
p\vee q, & (x,y)\in \left\{ (p,q),(q,p),(p,p),(q,q)\right\}
\end{cases}
\label{e1}
\end{equation}%
is an idempotent nullnorm on $L$ with the zero element $a$ if and only if $%
p\vee a=q\vee a.$
\end{corollary}

\begin{proof}
\textit{Necessity.}
Let the function $V_{5}:L\times L\rightarrow L$ defined
by the formula (\ref{e1}) be an idempotent nullnorm on $L$ with the zero
element $a.$ Without loss of generality, suppose that $p\leq q.$ In this
case, we have $V_{5}\left( 1,V_{5}\left( p,q\right) \right) =V_{5}\left(
1,p\vee q\right) =V_{5}\left( 1,q\right) =q\vee a$ and $V_{5}\left(
V_{5}\left( 1,p\right) ,q\right) =V_{5}\left( p\vee a,q\right) =\left( p\vee
a\right) \wedge \left( q\vee a\right) =p\vee a.$ This, together with the
associativity of $V_{5}$, implies that $p\vee a=q\vee a$.

\medskip

\textit{Sufficiency.} Let $p\vee a=q\vee a.$ Without loss of generality,
suppose that $p\leq q$. In this case, we have $\left( p\vee q\right) \wedge
\left( p\vee a\right) \wedge \left( q\vee a\right) =q\wedge \left( q\vee
a\right) =q.$ This, together with $p\vee q=q$, implies that $p\vee q=\left(
p\vee q\right) \wedge \left( p\vee a\right) \wedge \left( q\vee a\right) $.
From \cite[Theorem 3]{Ca2020-2}, it follows that the function $V_{5}:L\times
L\rightarrow L$ defined by the formula (\ref{e1}) is an idempotent nullnorm
on $L$ with the zero element $a.$
\end{proof}

Considering a bounded lattice $L$, $a\in L\setminus \left\{ 0,1\right\} $
and $I_{a}=\left\{ p,q\right\} $ such that $p\neq q$, $p\nparallel q$ and $%
p\vee a=q\vee a$, it should be pointed out that if $p\leq q$ (resp. $q\leq
p),$ then the nullnorm $V_{5}$ on $L$ given by the formula (\ref{e1})
coincides with the nullnorm $V_{4}$ (resp. $V_{3}$) on $L$ given by the
formula (\ref{eq-2.5}) (resp. (\ref{eq-1})).

The following Corollary \ref{cor2} can be proved in a manner similar to the
proof of Corollary \ref{cor1-1}.

\begin{corollary}
\label{cor2} Let $\left( L,\leq ,0,1\right) $ be a bounded lattice, $a\in
L\backslash \left\{ 0,1\right\} $ and $I_{a}=\left\{ p,q\right\} $ such that
$p\neq q$ and $p\nparallel q$. Then, the binary operation $V_{6}:L\times
L\rightarrow L$ defined by
\begin{equation}
V_{6}\left( x,y\right) =%
\begin{cases}
x\vee y, & (x,y)\in \lbrack 0,a]^{2}, \\
x\wedge y, & (x,y)\in \lbrack a,1]^{2}, \\
a, & (x,y)\in \left( \lbrack 0,a]\times \lbrack a,1]\right) \cup \left(
\lbrack a,1]\times \lbrack 0,a]\right) , \\
x\vee (y\wedge a), & (x,y)\in \lbrack 0,a]\times I_{a}, \\
y\vee (x\wedge a), & (x,y)\in I_{a}\times \lbrack 0,a], \\
x\wedge (y\vee a), & (x,y)\in \lbrack a,1]\times I_{a}, \\
y\wedge (x\vee a), & (x,y)\in I_{a}\times \lbrack a,1], \\
p\wedge q, & (x,y)\in \left\{ (p,q),(q,p),(p,p),(q,q)\right\}
\end{cases}
\label{e2}
\end{equation}%
is an idempotent nullnorm on $L$ with the zero element $a$ if and only if $%
p\wedge a=q\wedge a$.
\end{corollary}

Considering a bounded lattice $L$, $a\in L\backslash \left\{0, 1\right\}$,
and $I_{a}=\left\{p, q\right\}$ such that $p\neq q$, $p\nparallel q$ and $%
p\wedge a=q\wedge a$, it should be pointed out that if $p\leq q$ (resp., $%
q\leq p$), then the nullnorm $V_{6}$ on $L$ given by the formula (\ref{e2})
coincides with the nullnorm $V_{3}$ (resp., $V_{4}$) on $L$ given by the
formula \eqref{eq-1} (resp., \eqref{eq-2.5}).

\begin{remark}
\begin{enumerate}[{\rm (1)}]
\item Considering the commutativity of $V_3$ in the proof of Theorem~\ref%
{Wu-Operation}, the cases 2.1.1, 2.2.1, 2.3.1, 3.1.1, 3.1.2, 3.2.1, 3.2.2,
3.3.1, and 3.3.2 can be verified directly. We include them here for
completeness.
\item It should be noted that the last lattice in Fig.~\ref{FIG:1} is
isomorphic to the one depicted by Hasse diagram in \cite[Theorem 2]{CK2018}.
\end{enumerate}
\end{remark}

\begin{proposition}
\label{p1} Let $\left(L, \leq, 0, 1\right) $ be a bounded lattice, $a\in
L\backslash \left\{ 0,1\right\}$, and $I_{a}=\left\{ p,q\right\}$ with $%
p\neq q$. Then, ($p\vee a<q\vee a$ and $q\wedge a\leq p\wedge a$) or ($p\vee
a\leq q\vee a$ and $q\wedge a<p\wedge a$) if and only if there is only one
idempotent nullnorm on $L$ with the zero element $a$, which is given by the
formula \eqref{eq-1}.
\end{proposition}

\begin{proof}
{\it Necessity.} Let $p\vee a<q\vee a$ and $q\wedge a\leq p\wedge a$.
Then, for an idempotent nullnorm $V$ on $L$, it follows from Lemma~\ref{Ia-Operation}
that $V(p,q) \in \left\{p,q\right\}$. Suppose that $V(p,q)=q$. By the monotonicity of
$V$, for $q<1$, it holds $q=V(p, q) \leq V(p, 1) =p\vee a$. This
implies that $q\vee a\leq p\vee a$, which is a contradiction, and thus
$V(p, q)=p$. Therefore, by using Lemma~\ref{stru-lemma}, it follows that
there is only one idempotent nullnorm on $L$ with the zero element $a$, which is
given by the formula \eqref{eq-1}. Otherwise, if $p\vee a\leq q\vee a$ and
$q\wedge a<p\wedge a$, it can be verified similarly.

\medskip

{\it Sufficiency.}
Let $V$ be the only idempotent nullnorm on $L$ given by the formula \eqref{eq-1}.
In this case, $V(p, q)=p$. Suppose, on the contrary,
that ($p\vee a\nless q\vee a$ or $q\wedge a\nleq p\wedge a$) and
($p\vee a\nleq q\vee a$ or $q\wedge a\nless p\wedge a$).
From $V(p, q)=p$, it follows that $V(0, V(p, q))=V(0, p)=p\wedge a$ and
$V(V(0, p), q)=V(p\wedge a, q)=(p\wedge a)\vee (q\wedge a)$. This, together with
the associativity of $V$, implies that $q\wedge a\leq p\wedge a$. Furthermore,
$V(1, V(p,q))=V(1, p)=p\vee a$ and $V(V(1, p), q) =V( p\vee a, q)=(p\vee a) \wedge
(q\vee a)$. This, together with the associativity of $V$, implies that
$p\vee a\leq q\vee a$. From the supposition that ($p\vee a\nless q\vee a$ or $q\wedge a\nleq p\wedge a$)
and ($p\vee a\nleq q\vee a$ or $q\wedge a\nless p\wedge a$), it can be verified that
$p\vee a=q\vee a$ and $q\wedge a=p\wedge a$. This, together with the proof of Theorem~\ref{Wu-Operation},
implies that the binary operation $V_{4}:L\times L\rightarrow L$
defined by the formula \eqref{eq-2.5} is an idempotent nullnorm on $L$ with the zero
element $a$. This contradicts with the fact that $V$ is the only idempotent nullnorm on $L$.
\end{proof}

The following Proposition \ref{p2} can be proved in a manner similar to the
proof of Proposition \ref{p1}.

\begin{proposition}
\label{p2} Let $\left(L,\leq, 0, 1\right)$ be a bounded lattice, $a\in
L\backslash \left\{0,1\right\}$, and $I_{a}=\left\{p,q\right\}$ with $p\neq q
$. Then, ($p\wedge a<q\wedge a$ and $q\vee a\leq p\vee a$) or ($p\wedge
a\leq q\wedge a$ and $q\vee a<p\vee a$) if and only if there is only one
idempotent nullnorm on $L$ with the zero element $a$, which is given by the
formula \eqref{eq-2.5}.
\end{proposition}

\begin{proposition}
\label{p3} Let $\left( L,\leq ,0,1\right) $ be a bounded lattice, $a\in
L\backslash \left\{ 0,1\right\} $ and $I_{a}=\left\{ p,q\right\} $ such that
$p\neq q,$ $p\vee a\leq q\vee a$ and $q\wedge a\leq p\wedge a$. In that
case, the function $V_{4}:L\times L\rightarrow L$ given by the formula %
\eqref{eq-2.5} is an idempotent nullnorm on $L$ with the zero element $a$ if
and only if $p\vee a=q\vee a$ and $q\wedge a=p\wedge a.$
\end{proposition}

\begin{proof}
{\it Necessity.}
Suppose that the binary operation $V_{4}:L\times L\rightarrow L$ given by the formula \eqref{eq-2.5}
is an idempotent nullnorm on $L$ with the zero element $a$. Then $V_{4}(p,q)=q$. This implies that
$V_{4}(1, V_{4}(p, q))=V_{4}(1, q)=q\vee a$ and $V_{4}(V_{4}(1, p), q)=V_{4}(p\vee a, q)
=( p\vee a)\wedge ( q\vee a) =p\vee a$, and thus $p\vee a=q\vee a$ by the associativity of $V_{4}$.
Similarly, by applying $V_{4}(0, V_{4}(p, q))=V_{4}(V_{4}(0, p), q)$, we have $q\wedge a=p\wedge a$.

\medskip

{\it Sufficiency.}
Let $p\vee a=q\vee a$ and $q\wedge a=p\wedge a$. According to the proof of Theorem~\ref{Wu-Operation},
it is clear that the binary operation $V_{4}:L\times L\rightarrow L$ given by the
formula \eqref{eq-2.5} is an idempotent nullnorm on $L$ with the zero element $a$.
\end{proof}

The following Proposition \ref{p4} can be proved in a manner similar to the
proof of Proposition \ref{p3}.

\begin{proposition}
\label{p4} Let $(L, \leq, 0, 1)$ be a bounded lattice, $a\in L\backslash
\left\{0, 1\right\}$ and $I_{a}=\left\{p, q\right\}$ such that $p\neq q$, $%
p\wedge a\leq q\wedge a$ and $q\vee a\leq p\vee a$. In that case, the binary
operation $V_{3}:L\times L\rightarrow L$ given by the formula \eqref{eq-1}
is an idempotent nullnorm on $L$ with the zero element $a$ if and only if $%
p\vee a=q\vee a$ and $q\wedge a=p\wedge a$.
\end{proposition}

\begin{corollary}
\label{cor1} Let $(L, \leq, 0, 1)$ be a bounded lattice, $a\in L\backslash
\left\{0, 1\right\}$ and $I_{a}=\left\{p, q\right\}$ with $p\neq q$. In that
case, the binary operations both $V_{3}:L\times L\rightarrow L$ and $%
V_{4}:L\times L\rightarrow L$ given by the formulas \eqref{eq-1} and %
\eqref{eq-2.5}, respectively, are idempotent nullnorms on $L$ with the zero
element $a$ if and only if $p\vee a=q\vee a$ and $q\wedge a=p\wedge a$.
\end{corollary}

From Lemma~\ref{stru-lemma} and Corollary \ref{cor1}, we can easily observe
that on a bounded lattice $L$ with an element $a\in L\backslash \left\{0,
1\right\}$ satisfying that $I_{a}=\left\{p, q\right\}$, $p\neq q$, $p\vee
a=q\vee a$ and $q\wedge a=p\wedge a$, there exist only two idempotent
nullnorms $V_{3}:L\times L\rightarrow L$ and $V_{4}:L\times L\rightarrow L$
defined by the formulas \eqref{eq-1} and \eqref{eq-2.5}, respectively.

\begin{proposition}\label{Pro-special}
Let $(L, \leq, 0, 1)$ be a bounded lattice, $a\in L\backslash
\left\{0, 1\right\}$, and $p, q\in I_{a}$. If $p \nparallel q$, $p\wedge a=q\wedge a$ and $p\vee a=q\vee a$, then
$(p\wedge a)\vee (q\wedge a)\vee (p\wedge q)=p\wedge q$ and $(p\vee a)\wedge (q\vee a)\wedge (p\vee q)=p\vee q$.
\end{proposition}

\begin{proof}
Without loss of generality, assume that $p\leq q$ as $p\nparallel q$. This, together with
$p\wedge a=q\wedge a$ and $p\vee a=q\vee a$, implies that
$$
(p\wedge a)\vee (q\wedge a)\vee (p\wedge q)=(p\wedge a)\vee p=p=p\wedge q,
$$
and
$$
(p\vee a)\wedge (q\vee a)\wedge (p\vee q)=(q\vee a)\wedge q=q=p\vee q.
$$
\end{proof}

\begin{remark}
Proposition~\ref{Pro-special} shows that \cite[Theorem~2]{WZL2020} is a special case
of \cite[Theorems 3 and 4]{Ca2020-2}.
\end{remark}

\section{Conclusion}

Following the characterization of nullnorms on unit interval $\left[ 0,1%
\right] $, the characterization of nullnorms related to algebraic structures
on bounded lattices has recently attracted much attention. The concept of
nullnorms on bounded lattices was introduced by Kara\c{c}al et al. \cite%
{KIM2015}. In the meanwhile, they showed the existence of nullnorms on
bounded lattices with a zero element and constructed the greatest and the
smallest of them. After their demonstration, the constructions of nullnorms
on bounded lattices have become the focus of the study of many researchers.
Nevertheless the constructions of nullnorms on bounded lattices were firstly
provided by Kara\c{c}al et al. \cite{KIM2015}, these constructions do not generate an
idempotent nullnorm, in general. \c{C}ayl\i\ and Kara\c{c}al \cite{CK2018}
studied the idempotent nullnorms on bounded lattices and proved that an idempotent
nullnorm on a bounded lattice need not always exists. They also introduced a
construction method for idempotent nullnorms on bounded lattices,
where there is just one element incomparable with the zero element.
Afterward, \c{C}ayl\i\ \cite{Ca2020} proposed new construction approaches
for idempotent nullnorms on bounded lattices with some additional
assumptions. Notice that these methods encompass as a special case the one
introduced in \cite{CK2018}. Then, she \cite{Ca2020-2,Ca2020-3} introduced
some related constructions of nullnorms on bounded lattices with the zero
element $a$ based on the existence of triangular norms on $\left[ a,1\right]
^{2}$ and triangular conorms $\left[ 0,a\right] ^{2}$. Sun and Liu~\cite%
{SL2020} characterized a family of nullnorms on bounded lattices and
represented them by triangular norms and triangular conorms. In this paper, we continued to
study on idempotent nullnorms on bounded lattices with the zero element
different from the bottom and top elements from the mathematical point of
view. We presented the characterization of idempotent nullnorms on bounded
lattices, where there are just two distinct elements incomparable with the
zero element $a$. We also obtained some basic properties for the bounded
lattice containing two distinct elements incomparable with some element $a$.
We believe that the results presented in this paper can help to
understand the structure of the idempotent nullnorms on bounded lattices and
provide guidelines for a future work on their characterization.


\end{document}